\documentclass[reqno]{amsart}
\usepackage{amssymb}
\usepackage{amsfonts}
\usepackage{paralist}
\usepackage{enumitem}
\usepackage{indentfirst}
\usepackage{amsmath}
\usepackage[integrals]{wasysym}
\usepackage{amsthm}
\usepackage{mathrsfs}
\usepackage{standalone}
\usepackage{color}
\usepackage{adforn}
\usepackage{comment}%Omitir um text \begin{comment}text\end{comment}
\usepackage[bookmarks]{hyperref}
\usepackage{cleveref}
\usepackage{cancel}
\usepackage{faktor}
\usepackage{mdframed}%box de texto
\usepackage[normalem]{ulem}
\usepackage{todonotes}
\crefformat{section}{\S#2#1#3} % see manual of cleveref, ection 8.2.1
\crefformat{subsection}{\S#2#1#3}
\crefformat{subsubsection}{\S#2#1#3}
\crefrangeformat{section}{\S\S#3#1#4 to~#5#2#6}
\crefmultiformat{section}{\S\S#2#1#3}{ and~#2#1#3}{, #2#1#3}{ and~#2#1#3}

\newcommand{\T}{\mathbb{T}}
\newcommand{\R}{\mathbb{R}}
\newcommand{\Z}{\mathbb{Z}}

\newcommand{\TT}{\mathbb{T}^2}
\newcommand{\RR}{\mathbb{R}^2}
\newcommand{\ZZ}{\mathbb{Z}^2}
\newcommand{\half}{\frac{1}{2}}
\newcommand{\Leb}{\operatorname{Leb}}

\newcommand{\cF}{\mathcal{F}}

\newcommand{\cK}{\mathcal{K}}
\newcommand{\cA}{\mathcal{A}}
\newcommand{\tcA}{\tilde{\mathcal{A}}}

\newcommand{\tf}{\tilde{f}}
\newcommand{\tilh}{\tilde{h}}

\newcommand{\tU}{\tilde{U}}
\newcommand{\tgamma}{\tilde{\gamma}}

\newcommand{\tx}{\tilde{x}}
\newcommand{\ty}{\tilde{y}}

\newcommand{\tpsi}{\tilde{\psi}}

\newcommand{\tg}{\tilde{g}}
\newcommand{\tphi}{\tilde{\phi}}
\newcommand{\tp}{\tilde{p}}

\newcommand{\tq}{\tilde{q}}
\newcommand{\tldh}{\tilde{h}}
\newcommand{\cC}{\mathcal{C}}
\newcommand{\cU}{\mathcal{U}}
\newcommand{\hp}{\hat{p}}

\DeclareMathOperator{\SL2Z}{SL(2,\Z)}

\newtheorem{defi}{Definition}[section]
\newtheorem{thm}[defi]{Theorem}
\newtheorem*{thm*}{Theorem}
\newtheorem{prop}[defi]{Proposition} %[section]
\newtheorem{prob}[defi]{Problem}
\newtheorem{lemma}[defi]{Lemma}%[section]
\newtheorem{cor}[defi]{Corollary}%[section]
\newtheorem{rmk}[defi]{Remark}%[section]
\newtheorem{ex}{Example}%[section]
\newtheorem{theorem}{Theorem}

\newtheorem{cor-thm}[defi]{Corollary}%[theorem]

\newtheorem*{sc-lemma}{Semicontinuity Lemma}
\newcommand{\A}{\mathbb{A}}

\newcommand{\tcF}{\widetilde{\mathcal{F}}}
\newcommand{\cL}{\mathcal{L}}
\newcommand{\tGamma}{\tilde{\Gamma}}
\newcommand{\id}{\operatorname{id}}

\newtheorem*{corA}{Corollary A}
\newtheorem*{corB}{Theorem B'}
\begin{document}

\title{Partially hyperbolic endomorphisms with expanding linear part}

\author[M. Andersson]{Martin Andersson} 

\address{Martin Andersson: Instituto de Matem\'{a}tica Aplicada. Universidade
	Federal Fluminense. Rua Professor Marcos Waldemar de Freitas Reis, S/N.
24210-201 Niter\'{o}i, Brazil.}

\thanks{M. Andersson was supported by CNPq/MCTI/FNDCT project 403041/2021-0, Brazil}

\email{nilsmartin@id.uff.br}

\author[W. Ranter]{Wagner Ranter} 

\address{Wagner Ranter: Instituto de Matem\'{a}tica. Universidade Federal de
Alagoas, Campus A.S. Simoes S/N, 57072-090. Macei\'o, Alagoas, Brazil.}

\email{wagnerranter@im.ufal.br}

\thanks{W. Ranter was supported by CNPq/MCTI/FNDCT project 409198/2021-8, Brazil.}

%\date{\today}

\begin{abstract}
	In this paper we study transitivity of partially
	hyperbolic endomorphisms of the two torus whose action in the first
	homology has two integer eigenvalues of moduli greater than one. We prove that if
	the Jacobian is everywhere greater than the modulus of the largest eigenvalue, then
	the map is robustly transitive. For this we introduce Blichfedt's
	theorem as a tool for extracting dynamical information from the action
	of a map in homology.

	We also treat the case of specially partially hyperbolic endomorphisms,
	for which we obtain a complete dichotomy: either the map is transitive
	and conjugated to its linear part, or its unstable foliation must
	contain an annulus which may either be wandering or periodic.
\end{abstract}

\maketitle

%%%%%%%%%%%%%%%%%%%%%%%%%%%%%%%%%%%%%%%%%%%%%%%%%%%%%%%%%%%%%%%%%%%%%%%%%%%%%%%%%%%%%%%%%%%%%%%%%%%%%%%%%%%%%%%%%%%%%%%%%%%%%%%%%%%%%%%%%%%%%%%%%%%%%%%%%%%%%%%%%%%%%%%%%%%%%%%%%%%%%%%%%%%%%%%%%%%%%%%%%

\section{Introduction}

Although it may now be long forgotten, dynamicists once believed
that diffeomorphisms with gradient-like dynamics (so-called Morse-Smale
systmes) make up a dense subset among diffeomorphisms on any compact manifold.
That should remind us about how striking the existence of robustly transitive
diffeomorphisms actually is. Recall that a diffeomorphism $f$ is
\emph{transitive} if it has a dense orbit, and  \emph{robustly transitive} if
there is a $C^1$ neigbourhood $\cU$ of $f$ such that every $g \in \cU$ is
transitive. The first examples of robustly transitive diffeomorphisms were
Anosov diffeomorphisms, and for some time it was believed that there were no
others. But in the 70's, Shub and Ma\~{n}\'{e} gave examples of robustly
transitive diffeomorphisms on $\T^4$ and $\T^3$ respectively that are not
Anosov. Both of these examples are homotopic to Anosov (i.e.
``derived-from-Anosov'') and partially hyperbolic. Partial hyperbolicity is not
a necessary condition for robust transitivity, but an even weaker form of
hyperbolicity (dominated splitting with uniform contraction/expansion in the
extreme bundles, see \cite{DPU,BDP}) is. In particular, in dimension three, any
robustly transitive diffeomorphism must have a non-trivial dominated splitting
with uniform expansion or contraction in the one-dimensional bundle. Until the
90's there were no known examples of robustly transitive diffeomorphisms which
are not homotopic to Anosov. That changed with the publication of
\cite{BD}, where a new tool called blender was introduced, 
allowing for a whole range of new examples.
Yet it still remains an open problem to describe and
classify all robustly transitive derived-from-Anosov diffeomorphisms, even on
$\T^3$.

In hindsight it may seem surprising that the research on this topic was born in
the context of invertible maps, since the simplest examples of robustly
transitive maps are actually uniformly expanding maps.
It is therefore natural to ask whether it is possible to
describe and classify robustly transitive "derived-from-expanding" maps, i.e.
maps which are robustly transitive and homotopic to an
expanding map whithout being themselves expanding. In a sense, it is a more
elementary problem to classify derived-from expanding maps on, say, $\TT$ than
the analogous problem for derived-from-Anosov diffeomorphisms on
$\T^3$ and we believe that the former is the right starting point for both problems.
This is because of the simpler topology present in the derived-from-expanding
case.
In fact, there is a strong analogy between uniformly expanding
maps and Anosov diffeomorphisms which becomes apparent by lifting a uniformly
expanding map to its natural extension in the inverse limit space. 
Similarily, there is a strong analogy between derived-from-expanding maps on
$\TT$ and derived-from-Anosov maps with a dominated splitting and a uniformly
contracted one-dimensional bundle.

In spite of their more straightforward topological description, linear expanding maps on  $\TT$ come in a greater
variety than linear Anosov maps on $\T^3$. Whereas the latter must have either
three real irrational eigenvalues or one irrational and a pair of complex ones,
the former allows for a pair of irrational, a pair of complex, or a pair of
integer eigenvalues. This paper is dedicated to this latter case. 

\begin{prob}
	Fix a linear expanding map $A$ on $\TT$ with integer eigenvalues. What
	are the robustly transitive maps homotopic to $A$?
\end{prob}

Note that every homotopy class contains maps with
attractors, which is an obvious obstacle to transitivity, so the robustly
transitive maps cannot make up the whole homotopy class. Something extra is needed.  In previous works we have considered this question for maps which are conservative \cite{A} or for which the non-wandering set is the whole of $\TT$ \cite{WR}. Both conditions serve to make sure the map has no attractors and are in fact sufficient for transitivity. A possible candidate for a weaker condition would be maps which are \emph{volume expanding}.  Indeed, a volume expanding map cannot have an attractor whose trapping region is \emph{inessential}, i.e. which does not wind around the torus. But even volume expanding maps may have attractors with \emph{essential} trapping regions. 

\begin{ex} \label{invariant_stripe_example}
	Let $F$ be the direct product of
	two maps $f,g : S^1 \to S^1 $, where $f(x) = 3 x \mod 1$ and $g(x)$ a map
	homotopic to $x \mapsto 2x \mod 1$, satisfying
	\begin{enumerate}
		\item $g(0) = 0$
		\item $g'(0) <1$
		\item $  \frac{2}{3} < g'(x) < 3, \ \forall x \in S^1$.
	\end{enumerate}
	Then $F$ has Jacobian larger than $2$ everywhere but is
	clearly not transitive. Indeed, $g$ has an attractor at $0$, so $F$ has
	an attractor with trapping region of the form $S^1 \times (-\epsilon,
	\epsilon)$ for some $\epsilon>0$. Once an orbit enter this region, it
	cannot escape.
\end{ex}

Our main finding is that when the map is 
\emph{partially hyperbolic} and has a \emph{sufficiently large} Jacobian, then
it is robustly transitive. Let us be more specific.  

In this paper, an \emph{endomorphism} is synonymous with non-invertible local diffeomorphism. 
A \emph{partially hyperbolic} endomorphism is a local
diffeomorphism $f:\mathbb{T}^2 \to
	\mathbb{T}^2$ admitting an unstable cone-field $\cC^u: p \mapsto
	\cC_p^u$, where  $\cC_p^u$ is a closed cone in
$T_{p}\mathbb{T}^2$, and constants $\ell>0$ and $\lambda >1$ satisfying:
\begin{enumerate}[label=(\roman*)]
	\item $\cC^u$ is $Df^{\ell}$-invariant, that is,
		\[
			Df_p^{\ell}\cC_p^u \subseteq
			\operatorname{int}\cC^u_{f^{\ell}(p)}\cup \{0\}
		\]
		where
	      $\operatorname{int}(\cC_p^u)$ denotes the interior of
	      $\cC_p^u$;
	\item for every $v \in \cC_p^u, \ \ \|Df^{\ell}(v)\|\geq
		      \lambda\|v\|$.
\end{enumerate}

The action of an endomorphism in the first homology group is given by a $2
\times 2$  matrix with integer entries. We refer to this matrix (and the maps
it induces on $\RR$ and $\TT$) as the \emph{linear part} of the endomorphism.

\begin{theorem}\label{thm-A}
	Let $f:\TT \to \TT$ be a partially hyperbolic endomorphism whose linear
	part $A$ has	integer eigenvalues $\lambda_1, \lambda_2$ with
	$|\lambda_1|\geq |\lambda_2|>1$. Suppose that
	\begin{align}\label{sv-exp}
		|\det (Df_p)|>|\lambda_1| \ \ \text{for every p} \, \in
		\mathbb{T}^2.
	\end{align}
	Then $f$ is transitive.
\end{theorem}
Condition \eqref{sv-exp} says that the Jacobian of $f$ at every point is larger
than the spectral radius of the linear part of $f$. It can be slightly relaxed
by asking that it holds on an iterate of $f$ or, equivalently, that there is
some $C>0$ and $\lambda>\lambda_1$ such that $|\det(Df_p^n)| \geq C\lambda^n$
for every $n \geq 1$ and every $p \in \TT$. We say that an endomorphism with
this property is
\emph{strongly volume expanding}.

It should be noted that partial hyperbolicity and the strongly volume expanding condition are both persistent under $C^1$-perturbations.  As a consequence:

\begin{corA} \label{corollaryA}
	Suppose that $f$ is a partially hyperbolic endomorphism whose linear part is expanding with integer eigenvalues. If $f$ is strongly volume expanding, then $f$ is $C^1$ robustly transitive.
\end{corA}

Theorem \ref{thm-A} is similar in flavour to a theorem by Hertz, Ures and Yang
\cite{MR4419061} about partially hyperbolic diffeomorphisms on $\T^3$. Using
the hypothesis that $f$ is $C^2$ and a slightly weaker version of
\eqref{sv-exp} (they allow for equality in \eqref{sv-exp} in a set with zero
leaf volume along unstable leaves), they conclude that the strong stable and
unstable foliations are $C^1$ robustly minimal, which in particular implies
$C^1$ robust transitivity. Here we require less regularity but a
slightly stronger condition on the Jacobian than that of \cite{MR4419061}.
Notwithstanding the apparent similarities, the approaches taken in the two works
are very different. The argument in \cite{MR4419061} relies on the existence of
positive Lyapunov exponents in the center direction and makes thorough use of
the partially hyperbolic structure. In contrast, the present work applies
Blichfedt's Theorem to show that the strongly volume expanding conditions has
a rather far reaching topological consequence: a sufficiently high iterate of
any open set must wind around the torus in two directions (Lemma
\ref{doubly-essential}).  This is entirely
independent of the map being partially hyperbolic or not and is of independent
interest. Partial hyperbolicity is used to guarantee that this property indeed
implies transitivity.

\begin{comment}
\begin{ex}\label{ex1}
	$f:\mathbb{T}^2 \to \mathbb{T}^2, \ \ f(x,y)=(f_1(x),f_2(y)),$
	where $f_1(x)= 3x \,\, (\mathrm{mod} \,\, 1)$ and $f_2$ is a
	deformation of the doubling map, $x \mapsto 2x \,\,
		(\mathrm{mod} \,\, 1)$, around the fixed point $0 \in
		\mathbb{S}^1$ which $0$ becomes an attractor. Then, its linear
	part a diagonal matrix with $\lambda_1=3$ and $\lambda_2=2$.
	Since $\{0\}\times\mathbb{S}^1$ is attractor, we have that $f$
	is not transitive.
\end{ex}
\end{comment}

\subsection{Specially partially hyperbolic endomorphisms}
\label{special}
Whenever $f$ is a partially hyperbolic endomorphism, we may define the
\emph{center direction} at a point $p$ by
\[E_p^c = \{ v \in T_x \TT: Df_p^n( v) \notin \cC^u(f^n(p)) \ \forall n \geq 0
	\} \cup \{ 0\}.\]
However, in contrast to the invertible case, there
may not be a well defined unstable direction.
More precisely, given a choice of pre-orbit $\hp = ( \ldots, p_{-2}, p_{-1},
	p_0)$ of $p$, i.e. a sequence of points in $\TT$ satisfying $p_0 = p$ and
$f(p_{i-1}) = p_i$ for every $i \geq 0$, we define the direction
\begin{equation} \label{unstable_direction}
	\hat{E}_{\hp}^u  = \bigcap_{n \geq 0} Df^n(\cC^u(p_n)) .
\end{equation}

In general, $\hat{E}_{\hp}^u$ will depend on the particular choice of pre-orbit
$\hp$. In the exceptional case where it doesn't, we say that
$f$ is a \emph{specially partially hyperbolic endomorphism} and write $E_p^u =\hat{E}_{\hp}^u$. In this case, $E_p^u$ can easily be shown to be $f$-invariant
and continuous.

For specially partially hyperbolic endomorphisms we are able to give a full
characterization of transitivity both in terms of conjugacy and in terms of
absence of periodic or wandering annuli. By an \textit{annulus} we mean an open
subset $\mathbb{A}$ of $\mathbb{T}^2$ homeomorphic to $(-1,1)\times
S^1$. We say that an annulus $\mathbb{A}$ is \textit{periodic} if
there is $n\geq 1$ such that $f^n(\mathbb{A})=\mathbb{A}$; and  it is
\textit{wandering} if $f^n(\mathbb{A})\cap \mathbb{A}=\emptyset$ for every $n
\geq 1$. 

\begin{theorem}\label{thm-B} Let $f$ be a specially partially hyperbolic
	endomorphism with linear part $A$. Suppose that $A$ has integer eigenvalues
	$|\lambda_1| > |\lambda_2| >1$.  Then the following are equivalent:
	\begin{enumerate}[label=\rm{\alph*)}]
		\item $f$ is transitive; \label{f_transitive}
		\item $f$ topologically conjugated to $A$; \label{f_conjugated}
		\item $f$ admits neither a periodic nor a wandering annulus. \label{f_no_annulus}
	\end{enumerate}
\end{theorem}

When they exist, periodic and wandering annuli are necessarily saturated by
unstalbe leaves. We can therefore restate Theorem~\ref{thm-B} as:

\begin{corB} Let $f$ be a specially partially hyperbolic
	endomorphism with linear part $A$ having eigenvalues
	$|\lambda_1| > |\lambda_2| >1$.  Then one of the following holds:
	\begin{enumerate}[label=\rm{\alph*)}]
		\item $f$ is transitive and topologically conjugated to $A$;
			\label{not_trans}
		\item $f$ is not transitive and there is a periodic or wandering
			annulus saturated by the unstable foliation.
			\label{no_saturated_annulus}
	\end{enumerate}
\end{corB}

Note that, in virtue of being a direct product,
Example~\ref{invariant_stripe_example} is in fact specially partially
hyperbolic, so it serves as an example for the non-transitive case in
Theorems~\ref{thm-B} (and B'). In that example, the origin is an attractor for
$g$ whose basin is a union of intervals. If $I$ is the interval that contains
$0$, then $\T \times I$ is a periodic (in fact fixed) annulus.

%%%%%%%%%%%%%%%%%%%%%%%%%%%%%%%%%%%%%%%%%%%%%%%%%%%%%%%%%%%%%%%%%%%%%%%%%%%%%%%%%%%%%%%%%%%%%%%%%%%%%%%%%%%%%%%%%%%%%%%%%%%%%%%%%%%%%%%%%%%%%%%%%%%%%%%%%%%%%%%%%%%%%%

\subsection*{Acknowledgments}

We would like to thank Rafael Potrie and Enrique Pujals for their fruitful
comments suggestions.

%%%%%%%%%%%%%%%%%%%%%%%%%%%%%%%%%%%%%%%%%%%%%%%%%%%%%%%%%%%%%%%%%%%%%%%%%%%%%%%%%%%%%%

\section{Some Preliminaries } \label{section1}

An endomorphism $f: \TT \to \TT$ induces an action $f_\star$ on $\pi_1(\TT)$.
Since $\pi_1(\TT)$ is isomorphic to $\ZZ$, this action can e represented by a
$2 \times 2$ integer matrix $A$. Now, $A$ itself induces an endomorphism on
$\TT$, called a \emph{linear endomorphism}. Each endomorphism is homotopic to
one and only one such linear endomorphism, which we refer to as the
\emph{linear part} of $f$. One good reason for this is that if $\tf: \RR \to \RR$ is a
lift of $f$, then
\begin{equation}
	\tf(\tx+v) = \tf(\tx)+Av
	\label{lift_property}
\end{equation}
for every $\tx \in \RR$ and every $v \in \ZZ$. In particular, $\tf$ can be
neatly decomposed as $A + (\tf-A)$, where $\tf-A$ is $\ZZ$-periodic and hence
bounded. 

A linear map $A$ on $\RR$ is called \emph{expanding} when all its eigenvalues
have magnitude larger than one. In the case where the linear part $A$ of $f$
is expanding, there is a surjective
continuous map $h:\mathbb{T}^2
\to \mathbb{T}^2$, homotopic to the identity, such that
\begin{align} \label{semiconj}
	h\circ f= A\circ h.
\end{align}
The existence of $h$ was proved by Franks in \cite{MR0271990} for
diffeomorphisms with hyperbolic linear part, but the proof can be easily
adapted to endomorphisms with expanding linear part. (We remark that if the
linear part is a hyperbolic endomorphism, such a map may not exist. See
\cite{2104.01693}.) The map $h$ is called a \textit{semiconjugacy} from $f$ to
$A$. When $h$ is a homeomorphism we say that it is  a \textit{conjugacy}
between $f$ and $A$.

One of the consequences of the existence of the semi-conjugacy is that $\tf^n$
and $A^n \circ \tldh$ stay uniformly close. Indeed, if $\tldh$ is a lift of $h$,
then $\tldh - \id$ is $\ZZ$-periodic (since $h$ is homotopic to the identity)
and hence bounded by some constant, say $\kappa$. But $A^n (\tldh(\tx)) =
\tldh(\tf^n(\tx))$ so that 
\begin{equation}\label{bounded_distance}
	\|\tf^n(\tx)-A^n(\tldh(\tx))\|   <  \kappa
\end{equation}
for every $\tx \in \RR$ and every $n \geq 1$. 

It is sometimes useful to consider the set-valued function
\begin{align}
	\phi: \TT & \to \cK(\TT) \\
	x & \mapsto h^{-1}(h(x))
\end{align}
and its lift $\tphi(\tx) = \tldh^{-1}(\tldh(\tx))$. Here $\cK(\TT)$ denotes the class of compact subsets of $\TT$. The set $\tphi(\tx)$ is the set of points whose forward orbit stays a bounded distance away from the orbit of $\tx$ under iterations of $\tf$, i.e.
\[\tphi(\tx) = \{\ty \in \RR: \sup_{n \geq 0} \|\tf^n(\tx)-\tf^n(\ty)\|\} < \infty \}.\]

\begin{prop}\label{prop-Phi}
	Let $f: \TT \to \TT$ be an endomorphism with expanding linear part $A$
	and $\tf$ a lift of $f$. Then the following hold: 
	\begin{enumerate}[label=$\mathrm{(\alph*)}$]
		\item \label{characterization} There is $r>0$ such that
		      \[\tphi(\tx) =\bigcap_{k \geq 0}
			      \tf^{-n_k}(B(\tf^{n_k}(\tilde{x}),r)),\]
		      for each $\tilde{x} \in \mathbb{R}^2$ and each sequence
		      $n_k \to \infty$.
		\item \label{growing_ball} There exists $r_0$ and $k\geq 1$ such that
		      $\tf^k(B(\tx, r)) \supset \overline{B(\tf^k (\tx), r)}$
		      for every $\tx \in \RR$ and $r>r_0$, where $B(\tx, r)$
		      is the ball of radius $r$ centred at $\tx$.
		\item \label{tphi_connected} For each $\tilde{x} \in \mathbb{R}^2$,
		      $\tphi(\tilde{x})$ is a connected set.
		\item \label{tldh_connected} For each $\tx$, $\tldh^{-1}(\tx)$ is connected.
		\item \label{preimage_of_connected} For each compact connected
		      set $\mathcal{C}$ in $\mathbb{T}^2$, the set
		      $\tldh^{-1}(\mathcal{C})$ is connected.
	\end{enumerate}
\end{prop}

\begin{proof} The inclusion ``$\supset$'' in \ref{characterization} holds for
	every $r>0$. This follows by noting that iterates of any two points in
	the set on the right remain a bounded distance from one another. Since
	the linear part is expanding, this can only happen if they have the
	same image under $\tldh$.

	The inclusion ``$\subset$'' in \ref{characterization} holds for any
	$r>2\kappa$ where $\kappa>0$ is chosen such a way that
	$\|\tilde{h}-id\| \leq \kappa$. To see this, let $\ty \in \tphi(\tx)$.
	Then $\tldh(\ty) = \tldh(\tx)$ and, for $n\geq 0$,
	\begin{align*}
		\tldh(\tf^n(\tilde{y}))  =A^n(\tldh(\tilde{y}))
		=A^n(\tldh(\tilde{x})) =\tldh(\tf^n(\tilde{x})).
	\end{align*}

	Hence
	\begin{align*} \|\tf^n(\tilde{y})-\tf^n(\tilde{x})\| \leq
		\|\tf^n(\tilde{y})-\tldh(\tf^n(\tilde{y}))\| +
		\|\tldh(\tf^n(\tilde{x}))-\tf^n(\tilde{x})\| < r,
	\end{align*} and  we conclude that $\tilde{y} \in \bigcap_{n\geq 0}
		\tf^{-n}(B(\tf^n(\tilde{x}),r))$.

		Item \ref{growing_ball} holds because of
		\eqref{bounded_distance} and the fact that $A$ is expanding.

	To show \ref{tphi_connected}, fix $k$ and $r$ such that
	\ref{growing_ball} holds. If necessary, increase $r$ so that
	\ref{characterization} holds as well. Consider the sets $D_n(r) =
		\tf^{-n}(B(\tf^{n}(\tilde{x}),r))$. From \ref{characterization} we have
	that $\tphi(\tx) = \bigcap_{k \geq 0} D_{kn}$. Now,
	\[\tf^{k(n+1)}(\overline{D_{k(n+1)}}) = \overline{B(\tf^{k(n+1)}(\tx),
			r)} \subset \tf^k (B(\tf^{nk}(\tx), r)) = \tf^{k(n+1)}(D_{nk}),\] so
	that $\overline{D_{k(n+1)}}  \subset D_{nk}$. Hence $\tphi(\tx)$ can be
	written as $\bigcap_{n \geq 0} \overline{D_{nk}}$. In other words,
	$\tphi(\tx)$ is the intersection of
	a decreasing sequence of comact connected sets, so it is itself
	connected.

	Item \ref{tldh_connected} is an immediate consequence of \ref{tphi_connected}.

	We prove \ref{preimage_of_connected} by contradiction. First note that
	$\tldh^{-1}(\mathcal{C})$ is necessarily compact, since $\tldh$ is a
	bounded distance from the identity. Suppose that
	$\tldh^{-1}(\mathcal{C})$ is not  connected. Then there are disjoint
	compact sets $A$ and $B$ such that $\tldh^{-1}(\mathcal{C})=A\cup B$.
	Hence $\mathcal{C}=\tldh(A)\cup \tldh(B)$ with both $\tldh(A)$ and
	$\tldh(B)$ compact. Now, since $\mathcal{C}$ is connected, there exists
	some point $p \in \tldh(A) \cap \tldh(B)$. But then $\tldh^{-1}(p)$ can
	be written as the disjoint union $(\tldh^{-1}(p) \cap A) \cup
		(\tldh^{-1}(p) \cap B)$, both of which are closed. That is absurd.
\end{proof}

\begin{cor}\label{cor-connected} 
	Let $f: \TT \to \TT$ be an endomorphism with expanding linear part $A$. Then the following hold: 
	\begin{enumerate}[label=$\mathrm{(\alph*)}$]
		\item For each $p \in \mathbb{T}^2$, the set $h^{-1}(p)$ is a
		      connected set.
		\item For each closed connected set $\mathcal{C}$ in
		      $\mathbb{T}^2$, the set $h^{-1}(\mathcal{C})$ is
		      connected.
		\item For each $p \in \mathbb{T}^2$, $f(\phi(p))=\phi(f(p))$.
	\end{enumerate}
\end{cor}

\subsection{Dynamical coherence}

A partially hyperbolic endomorphism on $\TT$ is said to be \emph{dynamically
coherent} if there exists an invariant $C^0$ foliation with $C^1$ leaves tangent to $E^c$. When it
exists, such a foliation is called a \emph{center foliation} of $f$ and its
leaves are called \emph{center leaves}. If $f$ and $g$ are two dynamically
coherent partially hyperbolic endomorphisms, we say that $f$ and $g$ are
\emph{leaf conjugate} if there exists a homeomorphism $\psi: \TT \to \TT$
mapping center leaves of $f$ to center leaves of $g$. A \emph{periodic center
annulus} is an annulus $\A \subset \TT$ such that $f^n(\A) = \A$
for some $n \geq 1$ whose boundary consists of either one or two $C^1$ circles 
tangent to the center direction.

\begin{thm}[Hall and Hammerlindl \cite{HH-classification}] \label{HH-class}
	Let $f:\TT \to \TT$ be a partially hyperbolic endomorphism which does not admit a periodic center annulus. Then $f$ is dynamically coherent and leaf conjugate to $A$.
\end{thm}

\begin{rmk}
	In general, a partially hyperbolic endomorphism is not necessarily
	dynamically coherent, even when having expanding linear part. An
	example was given in \cite{HH-Incoherent} with linear part is as in 
	$\eqref{lower-triangular}$.
\end{rmk}

\subsection{Changing coordinates} 
This work concerns specifically endomorphisms whose linear part $A$ has
\emph{integer} eigenvalues.
It is convenient to suppose that one of the eigenspaces is the vertical
direction, i.e. that $A$ is represented by a lower triangular matrix of the form
\begin{equation} \label{lower-triangular}
	A = \left( \begin{matrix}
			\lambda_1 & 0         \\
			\mu       & \lambda_2
		\end{matrix} \right)
\end{equation}
where $|\lambda_1| \geq | \lambda_2| > 1$ are the (integer) eigenvalues of $A$
and $\mu$ is some integer.
There is no loss of generality in doing that.

\begin{lemma}\label{canonical_form} Let $A$ be a $2$ by $2$ matrix with integer entries and two integer	eigenvalues $\lambda_1, \lambda_2$. Then there exists $P \in	\SL2Z$ such that $P^{-1}A P $ is of the form \eqref{lower-triangular} for some $\mu \in \Z$.
\end{lemma}

\begin{proof}
	Since $A$ has integer eigenvalues, there exists $v \in \ZZ$ such that
	$Av = \lambda_2 v$. Without loss of generality, we may suppose that the
	components $v_1, v_2$ of $v$ are coprime. Let $p, q$ be such that $p
		v_1 + q v_2 = 1$ and take
	\begin{equation*}
		P = \left( \begin{matrix}
			q  & v_1 \\
			-p & v_2
		\end{matrix} \right).
	\end{equation*}
	Then $P^{-1}AP$ is of the form \eqref{lower-triangular}.
\end{proof}

%%%%%%%%%%%%%%%%%%%%%%%%%%%%%%%%%%%%%%%%%%%%%%%%%%%%%%%%%%%%%%%%%%%%%%%%%%%%%%%%%%%%%%%%%%%%%%%%%%%%%%%%%%%%%%%%%%%%%%%%%%%%%%%%%%%%%%%%%%%%%%%%%%%%%%%%%%%%%%%%%%%%%%%%%%%%

%%%%%%%%%%%%%%%%%%%%%%%%%%%%%%%%%%%%%%%%%%%%%%%%%%%%%%%%%%%%%%%%%%%%%%%%%%%%%%%%

\section{Proof of Theorem~\ref{thm-A}} \label{section3}

Before turning to the specific setting of Theorem~\ref{thm-A}, let us take a
look at how the strongly volume expanding property 
serves as a mechanism to produce homology in two linearly independent
directions for large iterates of an open set.

Recall that an open set $U \subset \TT$ is called \emph{essential} if it
contains a loop $\gamma$ such that its homotopy class $[\gamma]$ is non-zero in
$\pi_1(\TT) \cong \ZZ$. Similarily, we define $U$ to be \emph{doubly essential}
if it contains loops $\gamma$ and $\sigma$ such that $[\gamma]$ and $[\sigma]$
are linearly independent. 

It is straightforward to see that if $f$ is volume expanding, then a
sufficiently large iterate of any open set is essential. The main idea behind
Theorem~\ref{thm-A} is that strong volume expansion leads to high iterates of
any open set being doubly essential.

\begin{lemma} \label{doubly-essential}
	Let $f$ be a strongly volume expanding endomorphism on $\TT$.
 	Then, given any open set $U \subset \TT$, there exists $N \geq 0$ such that
	$f^n(U)$ is doubly essential for every $n \geq N$. 
\end{lemma}

The proof of Lemma \ref{doubly-essential} is a direct consequence of:

\begin{lemma} \label{doubly-essential-lift}
	Let $f$ be a strongly volume expanding endomorphism on $\TT$ and $\tf:
	\RR \to \RR$ a lift of $f$. Then, given any open set $\tU \subset \RR$,
	there exists $N \geq 0 $ such that for every $n \geq N$, there exist
	points $\tp_1, \tq_1, \tp_2, \tq_2$ in $\tf^n(\tU)$
	such that $\tp_1-\tq_1$ is a
	non-zero multiple of $e_1$ and $\tp_2-\tq_2$ is a non-zero multiple of
	$e_2$.
\end{lemma}

The proof of Lemma \ref{doubly-essential-lift} is based on a classical theorem about the geometry of numbers.

\begin{thm}[Blichfeldt's Theorem \cite{MR1500976}] \label{blichfeldt} 
	Let $B \subseteq \mathbb{R}^2$ be a Lebesgue measurable set such that
	$Leb(B) > k$ for some positive integer $k$. Then there exist $x_0,
	\ldots, x_k$ in $B$ such that $x_i-x_0 \in \mathbb{Z}^n$ for every $i=1,
	\ldots, k$.
\end{thm}

\begin{proof}[Proof of Lemma~\ref{doubly-essential-lift}]

Let $B$ be a (non-empty)  open connected
subset of $\tU$ contained in a ball of  radius less than one. By Gelfand's
formula,
\[
	\| A^n \|<(\lambda_1+\epsilon)^n
\]
for $n$ greater than some $n_0$. By \ref{bounded_distance} we have that $\tldh(B)$
is contained in a ball of radius $1+\kappa$ so that for $n> n_0$, 
$\tf^n(B)$ is contained in a ball of diameter less that $L_n =
2(1+\kappa)(\lambda_1+\epsilon)^n + 2 \kappa$. 
Choose $N>n_0$ so that $L_N  < \lambda^N \Leb(B)$. 

Now suppose that  $n \geq N$ and let $\ell$ be the integer part of $L_n$. 
Then
\[
	\Leb(\tf^n(B)) > \ell
\]
so by Blichfeldt's Theorem there is $\tx \in \RR$ such that $\tx+ \tf^n(B)$
intersects $\ZZ$ in at least $\ell+1$ points. 
Recall that  $L_n$ is an upper bound for the diameter of $\tf^n(B)$ so, upon possibly adding an element of
$\ZZ$ to $\tx$, we may assume that \[(\tx+ \tf^n(B)) \cap \ZZ \subset \{1,
\ldots, \ell\}^2.\]

	In other words, the intersection
	of $\tx+\tf^n(B)$ with $\ZZ$ consists of at least $\ell+1$ points
	and is contained in $\{1, \ldots, \ell \}^2$. By the pigeon hole
	principle there must be a line $\{1, \ldots, \ell\} \times \{i\}$
	containing two points $\tx_1, \ty_1$ of the intesection. Similarily,
	there is a column $\{j\} \times \{1, \ldots, \ell\}$ containing two
	points $\tx_2, \ty_2$ of the
	intersection. The proof follows by taking $\tp_i = \tx_i-\tx$ and $\tq_i =
	\ty_i-\tx $ for $i = 1, 2$.
\end{proof}

\begin{lemma} \label{leaf_conj_dyn_coh}
	  Let  $f: \TT \to \TT$ be a partially hyperbolic endomorphism. If $f$ is strongly volume expanding, then $f$ is  dynamically
	coherent and leaf conjugated to its linear part.
\end{lemma}

\begin{proof}
	By Theorem \ref{HH-class}, it suffices to show that $f$ does not admit
	a periodic center annulus. 
	Lemma~\ref{doubly-essential} implies that any open set must
	become doubly essential after a sufficient number of iterations. But no iterate of a periodic center annulus is doubly essential.
\end{proof}

\begin{rmk}
It is proved in \cite{HH-classification} that the absence of a periodic center annulus implies that the eigenvalues $\lambda_1$ and $\lambda_2$ of $A$ are distinct real numbers. 
\end{rmk}

In the proof of Theorem~\ref{thm-A} it will be convenient to reduce the
argument to the case in which $f$ is a skew-product. This can always be done
--- at least at the cost of sacrificing differentiability. Indeed, by
Lemma~\ref{leaf_conj_dyn_coh}, $f$ is leaf conjugated to its linear part $A$.
Let us denote the leaf conjugacy by $\psi$. Then the map $g = \psi \circ
f\circ\psi^{-1}$ preserve the foliation of $\TT$ into vertical circles (the
center leaves of the map $A$), and is therefore a skew product.

\begin{rmk} \label{homotopic_to_id}
	Although it is not stated explicitly in \cite{HH-classification}, it
	can be read from the proofs that the leaf conjugacy $\psi: \TT \to \TT$
	is	homotopic to the identity and $g = \psi f \psi^{-1}$ is of the
	form	$g(x,y) = (\lambda_1 x, \tau_x(y))$, where $\tau_x : S^1 \to S^1$
	is a continuous family of differentiable maps of degree
	$\lambda_2$. Since $\psi$ and $h$ are homotopic to the identity, so is
	$h_g$.
\end{rmk}

\begin{proof}[Proof of Theorem~\ref{thm-A}] Let $U \subset \TT$ be a
	(non-empty) open set. We shall show that there is some $n$ such that
	$f^n(U)  = \TT$.  We denote $\pi^{-1}(U)$ by $\tU$. Since $\tpsi(\tU)$
	is open, it contains an open rectangle $R = (x_1,x_2)\times(y_1, y_2)$.
	By Lemma~\ref{doubly-essential-lift} there exists $k$ such that
	$\tf^k(\tpsi^{-1}(R))$ contains points that differ by a non-zero
	multiple of $e_2$. But then the same is true for $\tg^k(R)$ (see
	Remark~\ref{homotopic_to_id}). 
	We are assuming $A$ to be of the form
	\eqref{lower-triangular} so that $\tg^k(R)$ is a union of vertical
	lines. This means that $\tg^k(R)$ must contain a vertical line whose
	length is larger than one. Since $\tg^n(R)$ is open, $\pi(\tg^n(R))$
	contains a vertical strip, i.e. a set of the form $I \times S^1$ for
	some open interval $I = (a,b)$. Iterating this strip $\ell$ times by
	$g$, where $|\lambda_1 |^\ell (b-a)>1$, we get the whole torus $\TT$. 
	The proof follows by taking $n=k+ \ell$. 
\end{proof}

\begin{rmk}
	The proof of Theorem~\ref{thm-A} shows that given any open $U \subset \TT$ there exists $n$ such
	that $f^n(U) = \TT$. This property, sometimes so called \emph{topological exactness},	or \emph{locally eventually onto} is much stronger than transitivity. In fact, it is straightforward to see that it implies topological mixing.	Hence Theorem~\ref{thm-A} and Corollary~\ref{corollaryA} remain valid if we replace `transitive' with `mixing'.
\end{rmk}

%%%%%%%%%%%%%%%%%%%%%%%%%%%%%%%%%%%%%%%%%%%%%%%%%%%%%%%%%%%%%%%%%%%%%%%%%%%%%%%%
%%%%%%%%%%%%%%%%%%%%%%%%%%%%%%%%%%%%%%%%%%%%%%%%%%%%%%%%%%%%%%%%%%%%%%%%%%%%%%%%
%%%%%%%%%%%%%%%%%%%%%%%%%%%%%%%%%%%%%%%%%%%%%%%%%%%%%%%%%%%%%%%%%%%%%%%%%%%%%%%%

\section{Proof of Theorem~\ref{thm-B}}\label{section2}
In what follows we shall fix a specially partially hyperbolic endomorphism \linebreak
$f:\mathbb{T}^2 \to \mathbb{T}^2$ and $\lambda_1, \lambda_2 \in \Z$ with $|\lambda_1|>|\lambda_2|>1$ as the eigenvalues of $A$.
Since the unstable direction (defined by \eqref{unstable_direction}) is independent of the past, 
$f$ has a non-trivial invariant splitting
\begin{equation}
	T_p\mathbb{T}^2=E^c\oplus E^u
\end{equation}
such that for all $p \in \TT$ and all unit vectors $v \in E^c_p$ and $w \in E^u_p$,
\begin{equation*} 
\| Df(v)\|<\|Df(w)\| \ \ \text{and} \ \ \|Df(w)\|>1.
\end{equation*}

Such an endomorphism always has a foliation tangent to the 
unstable bundle $E^u$. Indeed this follows by applying the classical arguments
of Hirsh, Pugh and Shub to the lift and then projecting to the torus (or
whatever be the manifold under consideration). Let us denote by $\mathcal{F}^u$
the foliation tangent to $E^u$ and call it the \textit{unstable foliation}. 

Although every specially hyperbolic endomorphism has an unstable foliation, it does not necessarily have a central one.
Indeed, in \cite{He-Shi-Wang} there is an example of a \emph{dynamically
incoherent} specially partially hyperbolic endomorphism (whose linear part is
not expanding). However, when the linear part is expanding, the next  result follows as direct consequence of \cite{HH-classification}[Theorem~E].

\begin{prop}
A specially partially hyperbolic endomorphism with expanding linear part does not admit a periodic center annulus. 
\end{prop}

By Theorem~\ref{HH-class}, $f$ is dynamically coherent and leaf conjugate to $A$. We fix $\cF^c$ as the center foliation. Let $E_A^u$ and $E_A^c$ be the eigenspaces corresponding to $\lambda_1$ and $\lambda_2$ respectively. We denote by $\tcA^u$ and $\tcA^c$ the foliations of $\RR$ by lines parallel to these spaces and by $\cA^u$ and $\cA^c$ the foliations they induce on $\TT$.

We denote by $\pi^u$ is the projection to $E_A^u$ whose whose kernel is $E^c_A$ and $\pi^c$ is the projection to $E_A^c$ whose kernel is $E_A^u$. We say that a foliation $\cF$ in $\RR$ is at a bounded distance from $\cA^c$ (respectively $\cA^u$) if there is some $M>0$ such that the length of $\pi^u(\mathcal{L})$ (resp. $\pi^c(\mathcal{L})$) is smaller than $M$ for every $\mathcal{L} \in \cF$.

Since the eigenvalues of $A$ are integers, $\cA^u$ and $\cA^c$ consist of circles. In particular, we also have that all the leaves of the center foliation $\cF^c$ of $f$ are also circles and, moreover, the leaves of $\tcF^c$ are at bounded distance from the lines of $\cA^c$.

As explained in \cite[Section 4.A]{Potrie}, every leaf of $\tcF^u$ is at a
bounded distance from a linear foliation on $\mathbb{R}^2$.  Since $\tcF^u$
is $f$-invariant, this foliation is $A$-invariant.
In our setting, there are two such foliations to choose from, i.e. $\tcA^c$ and
$\tcA^u$, and we need to take a closer look at $\tcF^u$ in order to see that only
the latter is possible. Similarily we will show that $\tcF^c$ is at a bounded
distance from $\tcA^c$.

Two
important concepts for
understanding foliations on $\mathbb{T}^2$ are Reeb components and Tannuli.
A \emph{Reeb component} of a foliation $\cF$ on $\TT$ is an annulus $\A$ such
that the restriction of $\cF$ to the closure of  $\A$ is homoeomorphic to one of the following:

\begin{enumerate}
	\item the foliation on $[-1,1]\times \mathbb{S}^1$
		induced by the foliation on $[-1, 1] \times \R$ given by the
		lines $\{-1\} \times \R$ and $\{1\} \times \R$, along with the
		graphs of the functions $ x \mapsto \exp(1/(1-x^2))+y$ with $y
		\in \mathbb{R}$. 
	\item the foliation on $\TT$ induced by the foliation on $S^1 \times
		\R$ obtained by identifying $\{-1\} \times \R$ with $\{1\} \times
		\R$ in case (1).
\end{enumerate}

A \emph{Tannulus component} (or simply \emph{tannulus}) is defined analogoulsy,
replacing the functions $x \mapsto \exp(1/(1-x^2))+y$ with $x \mapsto
\tan(\pi x/2) + y$. See Figures~\ref{fig.Reeb} and \ref{fig.tan}. 

By the classification of foliations on $\TT$ (see
\cite[Proposion~4.3.2]{Foliation-partA}), if a foliation does not admit Reeb
components then it is a suspension of a circle homeomorphism. Such a foliation may or may not contain a tannulus component.

\begin{figure}[htb!]
	\begin{minipage}[b]{0.45\linewidth}
		\centering
		\includegraphics[scale=0.8]{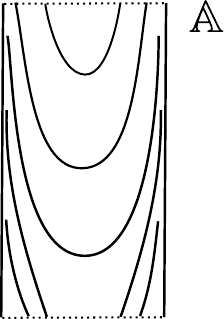}
		\caption{Reeb component}
		\label{fig.Reeb}
	\end{minipage}
	\begin{minipage}[b]{0.45\linewidth}
		\centering
		\includegraphics[scale=0.8]{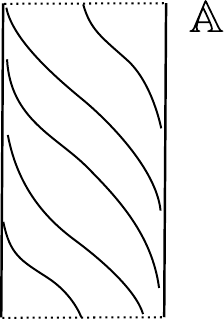}
		\caption{Tannulus}
		\label{fig.tan}
	\end{minipage}
\end{figure}

\begin{rmk}
	A foliation on $\mathbb{T}^2$ may have 
	infinitely many tannuli but it can have at most
	finitely many Reeb components. See \cite{Foliation-partA}.
\end{rmk}

%\begin{prop}% \label{unstable_tannuli_vs_DC}
%If $\cF^u$ does not admit a tannulus, then $f$ is dynamically coherent.	Moreover, $h$ sends leaves of $\cF^c$ onto leaves of $\cA^c$ and leaves	of $\cF^u$ onto leaves of $\cA^u$. 
%\end{prop}

A main ingredient is the following very general topological lemma. 

\begin{lemma} \label{no_periodic_annulus}
Let $f: \TT \to \TT$ be a self-cover. If there exists an annulus $\A$ and $n \geq 1$ such that  $\A = f^{-n}(\A)$, then the linear part of $f$ has an eigenvalue $\pm 1$.
\end{lemma}
Since we are assuming that $f$ has expanding linear part, Lemma
\ref{no_periodic_annulus} implies that there cannot be a \emph{backward}
invariant annulus.

The proof of Lemma~\ref{no_periodic_annulus} follows
by the arguments used in \cite{A} and
\cite{WR-thesis}. In short, if $\A$ is a periodic annulus with $f^{-n}(\A)  =
\A$, then the restriction of $f^n$ to $\A$ is a self-cover of degree
$\lambda_1^n \cdot \lambda_2^n$. At the same time, if $i: \A \to \TT$ is the
inclusion map, then  $i_\star$ sends the fundamental group of $\A$ to a
subgroup of $\ZZ$ of the form $G = \{k v: k \in \Z\} \subset \ZZ$, where $v \in
\ZZ$ is an eigenvalue of the linear part of $f$. The action $f$ on $G$ produces
a subgroup whose index is on the one hand equal to $\lambda_1^n \cdot
\lambda_2^n$, and on the other equal to $\lambda_i^n$, where $\lambda_i$ is the
eigenvalue associated to $v$. Hence the other eigenvalue must be $\pm 1$.

%The proof of Proposition \ref{unstable_tannuli_vs_DC} consists of five steps,organized here as Lemmas~\ref{u-foliation} to \ref{lema:leaf}.

Next, it is showed that $\cF^u$ is necessarily a suspension. 

\begin{lemma}\label{u-foliation}
The unstable foliation $\mathcal{F}^u$ has no Reeb component. 
\end{lemma}

\begin{proof}
	Suppose by contradiction that $\mathcal{F}^u$ contains a Reeb component
	$\mathbb{A}\subseteq \mathbb{T}^2$. Then, by
	\cite[Lemma~2.2]{He-Shi-Wang}, there is an integer $n>0$ such that
	$f^{-n}(\mathbb{A})=\mathbb{A}$.
	But that is impossible according to Lemma~\ref{no_periodic_annulus},
	since we are assuming that $f$ has expanding linear part.
\end{proof}

As we mentioned above, it follows from the classification of foliations on $\mathbb{T}^2$ 
that $\cF^u$ is a suspension.
Moreover, $\tcF^u$ has rational slope since its
leaves are a bounded distance from an eigenspace of $A$. Thus by the
classification of foliations on $\mathbb{T}^2$, either $\cF^u$ has a tannulus
or all the leaves of $\cF^u$ are circles.

\begin{lemma} \label{curva_transversal}
Let $\cF$ be a foliation of $\TT$ in which every leaf is a
circle. Then every leaf  of $\cF$ represents the same non-zero element
$v$ in $\ZZ$ (the fundamental group of $\TT$). Suppose, moreover, that $\gamma$ is a closed $C^1$ curve transverse to $\cF$. Then $[\gamma]$ is not a multiple of $v$.
\end{lemma}

\begin{proof}
	Let $\mathcal{L}$ be a leaf of $\cF$ and write $v = [\cL]$. That $v$
	is non-zero can be deduced from the Poincaré-Benedixon Theorem (a
	foliation of $\RR$ cannot have a compact leaf). If $\cL'$ is another
	leaf then $[\cL']$ must be equal to $v$, for else
	$\cL$ and $\cL'$ would intersect.
	Fix some lift $\tgamma:[0,1] \to \RR$ of $\gamma$ and extend it
	periodically to $\tGamma:\R \to \RR$. We claim that $\mathcal{L}$
	intersects (the image of) $\tGamma$. Indeed, this also follows from the
	Poincaré-Benedixon Theorem  since if it were not true, then the vector field
	tangent to $\tcF$ would exhibit a singularity. 

	We now observe that $\tGamma(t+k)= \tGamma(t) + k [\gamma]$ for every $k \in \ZZ$ so that the image of  $\tGamma$ is invariant under translation by $[\gamma]$. 
Similarily, $\mathcal{L}$ is invariant by translation of $v$.  Hence $[\gamma]$ cannot
be a multiple of $v$. For if it were then $\mathcal{L}$ and  $\tGamma$ would have infinitely many intersections.
\end{proof}

\begin{lemma} \label{if_no_tannuli}
%Suppose that the foliation $\cF^u$ has no tannulus. Then $\cF^c$ and $\cF^u$ are foliations by circles. Their lifts $\tcF^c$ and $\tcF^u$ are a bounded distance from $\tcA^c$ and $\tcA^u$, respectively.
The lifts $\tcF^c$ and $\tcF^u$ are a bounded distance from $\tcA^c$ and $\tcA^u$, respectively.
\end{lemma}

\begin{proof}
Recall that every leaf of $\tcF^u$ is a bounded distance from a translation of an
eigenspace of $A$. Since $\cF^u$ has a tannulus or all its leaves are circles, it is known that in both cases there is a circle as a leaf. Then, as such circle of $\cF^u$ is transverse to $\cF^c$, we can conclude by Lemma~\ref{curva_transversal} that this eigenspace cannot be $E_A^c$.  So it has to be $E_A^u$.
\end{proof}

%\begin{proof}
%If $\cF^u$ has no tannulus, then the leaves of $\cF^u$ are circles and, by Lemma~\ref{no_center_annulus}, $f$ has no center annulus. Hence, by Theorem~\ref{HH-class}, we conclude that the  leaves of $\mathcal{F}^c$ are circles.  Moreover, by \cite{HH-classification}[Lemma~3.8], they are at a bounded distance from the circles in $\cA^c$.

%Recall that every leaf of $\tcF^u$ is a bounded distance from a translation of a eigenspace of $A$. Then, as the leaves of $\mathcal{F}^u$ are circles transverse to $\mathcal{F}^c$, we can conclude by Lemma~\ref{curva_transversal} that this eigenspace cannot be $E_A^c$.  So it has to be $E_A^u$. 
%\end{proof}

A consequence of Lemmas~\ref{if_no_tannuli} is that the restriction of $\pi^c$
(resp. $\pi^u$) to $\tcF^c(\tp)$ (resp. $\tcF^u(\tp)$) is onto, so
$\tcF^c(\tp)$ and $\tcF^u(\tp)$ intersect each other. By the Poincaré-Bendixson
Theorem, we conclude that they intersect each other exactly
once. In other words, $\tcF^c$ and $\tcF^u$ have global product structure
and are \textit{quasi-isometric}. That is,
\begin{align}\label{eq1}
	\exists a,b>0 \ \ \text{such that} \ \ d_{\tcF^{\ast}}(\tp,\tq)\leq
	a\|\tp-\tq\|+b,
\end{align}
where $d_{\tcF^{\ast}}(\tp,\tq)$ denotes the distance between $\tp$ and $\tq$
along of a leaf of $\tcF^{\ast}$, for $\ast = c, u$.

%To conclude the proof of Proposition~\ref{unstable_tannuli_vs_DC}, it remains to show the following.

\begin{lemma}\label{lema:leaf}
The map $\tldh$ sends leaves of $\tcF^c$ onto leaves of $\tcA^c$ and leaves of $\tcF^u$ onto leaves of $\tcA^u$. %Moreover, the restriction of $\tldh$ to any leaf in $\tcF^u$ is a bijection onto its image.
\end{lemma}

\begin{proof}
Since  $\tcF^c$ is at a bounded distance from $\tcA^c$, there is a
constant $R>0$ such that for every $\tp \in \mathbb{R}^2$ we can find a
line $\mathcal{L} \in \tcA^c$ such that the leaf $\widetilde{\mathcal{F}}^c(\tp)$ is contained in $R$-neighbourhood of $\mathcal{L}$, which is an $R$-vertical strip. By \eqref{bounded_distance}, we have that $\|A^n\circ \tilh-\tf^n\|<\kappa$ for each integer $n$ and, thus, $A^n(\tilh(\widetilde{\mathcal{F}}^c(\tp)))$ is	contained in an	$(R+\kappa)$-vertical strip. 

Now, suppose that $\tq \in \tcF^c(\tp)$ and that $\tldh$ sends $\tp$ and $\tq$ to $(x_1, x_2)$ and $(y_1, y_2)$ in $\RR=E^u_A\oplus E^c_A$ respectively, with $x_1\neq y_1$. (Recall that we are assuming $A$ to be of the form \eqref{lower-triangular}, so that $\cA^c$ consists of vertical lines.) Then
	\[
		|\pi^u(A^n(x_1, x_2)) - \pi^u(A^n(y_1, y_2))| =|\lambda_1|^n |x_1-y_1|
	\] 
	gets arbitrarily large as $n$
	grows,	contradicting that $A^n(\tldh(\tcF^c(\tp)))$ is contained in
	a	$(R+\kappa)$-vertical strip. That proves that $\tldh$ sends leaves of
	$\tcF^c$ to lines in $\tcA^c$.
	The case of $\tcF^u$ is identical.
\end{proof}
	
%Lemmas~\ref{no_center_annulus} and \ref{lema:leaf} conclude the proof of Proposition~\ref{unstable_tannuli_vs_DC}.

\begin{lemma}\label{lemma-phi}
%Suppose that $\cF^u$ has no tannulus. Then the map $\tilh$ sends distinct leaves of $\tcF^c$ to distinct lines of $\tcA^c$.
The map $\tilh$ sends distinct leaves of $\tcF^c$ to distinct lines of $\tcA^c$.
\end{lemma}

\begin{proof}
	We argue by contradiction. Suppose there are distinct leaves, say
	$F_1$ and $F_2$, of $\tcF^c$ which are sent to the same line by $\tldh$.
	Then for every $\tq_1 \in F_1$ and every $\tq_2 \in F_2$ we have
	$\pi^u(\tilh(\tq_1))=\pi^u(\tilh(\tq_2))$ and so
	$\|\pi^u(\tf^n(\tq_1))-\pi^u(\tf^n(\tq_2))\|$
	is bounded for $n \geq 0$. By the global product structure, we
	can choose $\tq_1$ and $\tq_2$ in the same leaf of $\tcF^u$. Since
	$\tcF^u$ is at a bounded distance from $\tcA^u$, we have that
	$\|\pi^c(\tf^n(\tq_1))-\pi^c(\tf^n(\tq_1))\|$ is also bounded
	for $n \geq 0$. Hence $\|\tf^n(\tq_1)-\tf^n(\tq_2)\|$ is
	bounded for $n \geq 0$. But that is impossible since $\tq_1$ and $\tq_2$
	are in the same unstable leaf which is quasi-isometric. 
\end{proof}

A consequence of Lemma~\ref{lemma-phi} is that $\tphi(\tp)$ is contained in
$\tcF^c(\tp)$ for every $\tp \in \RR$. Proposition~\ref{prop-Phi} then implies
that $\tphi(\tp)$ must be eiter a point or a compact line segment in
$\tcF^c(\tp)$.

\begin{lemma}\label{cor:stripe}
Suppose that $\cF^u$ has no tannulus. If $\phi(p) \neq \{p\}$, then the interior of $h^{-1}(\cA^u(h(p)))$ is an annulus which is either wandering or periodic for $f$.
\end{lemma}

\begin{proof}
Since $\cF^u$ has no tannulus, the leaves of $\cF^u$ are cirlces so we may consider fibers of a trivial bundle $ \pi : \TT \to S^1$ whose fibers are the leaves of $\cF^u$. The set $\phi(p)$ is a transversal segment to the fibers and $h$ sends $\cF^u(x)$ to $\cA^u(h(p))$ for every $x \in \phi(p)$. Hence $h^{-1}(\cA^u(h(p))$ is equal to $\pi^{-1}(\pi(\phi(p)))$.
\end{proof}

\begin{proof}[Proof of Theorem~\ref{thm-B}]
The implication \ref{f_conjugated} $\Longrightarrow$ \ref{f_transitive} is
obvious. To see why \ref{f_transitive} $\Longrightarrow$
\ref{f_no_annulus}, first note that a transitive map may not have a wandering
open set of any kind. Suppose that $f$ has a periodic annulus $\A = f^n(\A)$ for
some $n\geq 1$. Then, by transitivity of $f$, we must have $f^{-n}(\A) = \A$. Indeed, if it were not so, $f^{-n}(\A)$
would consist of a union of several annuli, some of which would be wandering.
But Lemma~\ref{no_periodic_annulus} says that it is impossible to have a backward invariant annulus when the linear part is expanding.

It remains to show that \ref{f_no_annulus} implies \ref{f_conjugated}.
Note that $h$ is a conjugacy between $f$ and $A$ if and only if $\phi(p) = \{p\}$ for every $p\in \TT$. (A continuous bijection on a compact space is a homeomorpism.) Thus, by Lemma~\ref{cor:stripe}, it suffices to show that if $f$ does not admit a wandering or periodic annulus, then $\cF^u$ does not admit a tannulus. Suppose it does admit a tannulus $\A$. Then $f^n(\A)$ would be a tannulus for every $n\geq 0$. Moreover, $\A$ and $f^n(\A)$ must either coincide or be disjoint. Hence $\A$ must be either wandering or periodic.
\end{proof}

\section{An example}   Here we present a non-trivial example of an endomorphism
satifying the hypotheses of Theorem \ref{thm-A}. More precisely, we  construct
a $C^\infty$ local diffeomorphism $f: \TT \to \TT$ satisfying
\begin{enumerate}
	\item the linear part of $f$ is $A = \left( \begin{smallmatrix} 5 & 0
		\\ 0 & 2 \end{smallmatrix} \right)$,
	\item $\det Df(x,y)>5$ for every $(x,y) \in \TT$,  \label{sve}
	\item $f$ is partially hyperbolic, and
	\item $f$  has a hyperbolic fixed point with stable index $1$ and is
		therefore $C^1$ persistently not conjugated to $A$.
		\label{not_conjugated}
		\end{enumerate}
		By Theorem~\ref{thm-A}, $f$ is robustly transitive. The example
		is a skew-product, but all properties are  robust, so the
		construction leads implicitly to examples which are not
		skew-products. They are, however, topologically conjugated to
		skew-products. But that is unavoidable according to
		\cite{HH-classification} (see Theorem \ref{HH-class}).

Here's the construction. Let $\alpha: \T \to \T$ and $\beta: \TT \to \T$ be
given by 
\begin{align}  
	\alpha(x)  &= 5x + \frac{\sin(2 \pi x) }{2 \pi}  \\ 
	\beta(x,y)  &= 2y - (1+\epsilon)\cos^{2}(\pi x) \frac{ \sin(2 \pi y)}{2 \pi}
\end{align}
and take $f(x,y)   = (\alpha(x), \beta(x,y))$. Clearly $f$ is a well defined
$C^\infty$ map on $\TT$ homotopic to $A$. That it is a local diffeomorphism
will follow as soon as we have proved item (\ref{sve}) above.  The derivative
of $f$ at $(0,0)$ is given by
\[\left(\begin{matrix}
		6 & 0 \\
		0 & 1-\epsilon
\end{matrix}\right)\] which is hyperbolic with stable index $1$ for every
$\epsilon>0$. This property persists under $C^1$ perturbations and guarantees
that neither $f$ nor its neighbours are conjugated to $A$. To se why 
(\ref{sve}) holds, note that the
Jacobian 
\begin{equation*}
	J(x,y) =  | \det Df(x,y)| = \left( 5+\cos(2 \pi x) \right)
	\left( 2-(1+\epsilon)\cos^2(\pi x)\cos(2\pi y)\right)
\end{equation*}
is $C^\infty$ on $\TT$ and that 
\begin{equation*}
	\partial_y J = 2 \pi (1+\epsilon)(5+\cos(2 \pi x))
	\cos^2(\pi x) \sin(2 \pi y)
\end{equation*}
vanishes only on $x=\half$, $y=\half$, and $y=0$. It therefore suffices to
check that $J$ is greater than $5$ along these three curves.
\begin{itemize}
	\item On $x = \half$ we have $J(\half, y) \equiv 8$.
	\item On $y=\half$ we have $J(x,\half) = \left( 5+\cos(2 \pi x) \right)
		\left(2+(1+\epsilon)\cos^2(\pi x) \right) \geq 8$.
	\item On $y=0$ we have \begin{align*} J(x,0) & =
			\left(5+\cos(2 \pi x) \right)
			\left( 2-(1+\epsilon) \cos^2(\pi x) \right)     \\
		      & = 6+2 \sin^2(\pi x) (2- \sin^2(\pi x))
		      - \epsilon \cos^2(\pi x) (5+\cos(2 \pi x))      \\
		      & \geq 6- 6 \epsilon,
		\end{align*}
		which is greater that $5$ for every $\epsilon < 1/6$. That
		proves (\ref{sve}).
	\end{itemize}

Finally let us verify that $f$ is partially hyperbolic. For that, fix
some $p \in \TT$ and let  $(u_1, u_2)= Df_p (1,1)$, $(w_1,w_2) = Df_p (1,-1)$.
We claim that
\begin{equation}
	u_1  = w_1 \geq 4 \label{u_1},
\end{equation}
\begin{equation}
	1-\epsilon \leq u_2  \leq 3+\epsilon, \label{u_2}
\end{equation}
and
\begin{equation}
	-3-\epsilon \leq w_2  \leq 1+\epsilon.
	\label{w_2}
\end{equation}
Once that is shown, it follows that the cone
\begin{equation*}
	S = \{ (v_1, v_2) \in \RR\setminus \{(0,0)\} : |v_1| \geq |v_2| \}
\end{equation*}
is strictly $Df_p$-invariant at every $p \in \TT$ as long as
$\epsilon<1$.
The estimate in \eqref{u_1} also shows that vectors in $S$ are expanded
by $Df_p$ by a factor of at least $2 \sqrt{2}$. This is because
\begin{equation*}
	\max_{0 \leq t \leq 1} \|t(1,1) + (1-t)(1,-1)\| = \sqrt{2}
\end{equation*}
while
\begin{equation*}
	\min_{0\leq t \leq 1} \|t Df_p (1,1) + (1-t) Df_p (1,-1)\|
	\geq 4
\end{equation*}
for every $p$, and every $v \in S$ is a multiple of a vector of this
type.

It remains to prove \eqref{u_1}, \eqref{u_2}, and \eqref{w_2}.
For that, let us write $p = (x,y)$. 
Then inequality \eqref{u_1} is immediate, as
\begin{equation}
	u_1 = w_1 = \partial_x \alpha(x) = 5+\cos(2 \pi x).
	\label{horizontal_expansion}
\end{equation}
The inequalities in \eqref{u_2} follows by rewriting $u_2$ as
\begin{align*}
	u_2 & = \partial_x \beta(x,y) + \partial_y \beta(x,y)               \\
	    & = \left(\epsilon + 1\right) \sin{\left(\pi x \right)}
	\sin{\left(2 \pi y \right)} \cos{\left(\pi x \right)}               \\
	    & + 2 - \left(\epsilon + 1\right) \cos^{2}{\left(\pi x \right)}
	\cos{\left(2 \pi y \right)}                                         \\
	    & = 2- ( \cos(2 \pi y) + \cos(2 \pi (x+y)))/2                   \\
	    & - \epsilon \cos(\pi x) \cos(\pi(x+2y)).
\end{align*}
One can rewrite $w_2$ in a similar fashion to obtain \eqref{w_2}.

%%%%%%%%%%%%%%%%%%%%%%%%%%%%%%%%%%%%%%%%%%%%%%%%%%%%%%%%%%%%%%%%%%%%%%%%%%%%%%%%%%%%%%%%%%%%%%%%%%%%%%%%%%%%%%%%%%%%%%%%%%%%%%%%%%%%%%%%%%%%%%%%%%%%%%%%%%%%%%%%%%%%%%%%%%%%%%%%%%%%%%%%%%%%%%%%%%%%%%%%%%%%%%%%%%%%%%%%%%%%%%%%%%%%%%%%%%%%%%%%%%%%%%%%%%%%%%%%%%%%%%%%%%%%%%%%%%%%%%%%%%%%%%%%%%%%%%%%%%%%%%%%%%%%%%%%%%%%%%%%%%%%%%%%%%%%

%% References **************************************************************************************************************************************

\bibliographystyle{alpha}
\bibliography{Thesis-Biblio}

\end{document}